\documentclass[a4paper, oneside, english, reqno]{amsart}

\usepackage[utf8]{inputenx}
\usepackage{amsthm, amsmath, amssymb, amsfonts, mathrsfs, mathtools, stmaryrd}
\usepackage{textcomp, url, enumerate, latexsym, graphicx, varioref, hyperref, multirow, layout, siunitx, booktabs}
\usepackage{pdflscape}
\usepackage{verbatim}
\usepackage{geometry}
\usepackage{float}
\usepackage{babel}
\usepackage{csquotes}
\usepackage{tikz, rotating, tikz-cd, pigpen}
\usetikzlibrary{matrix, arrows, decorations.pathmorphing}
\usepackage{cleveref}
\RequirePackage[color       = blue!20,
                bordercolor = black,
                textsize    = tiny,
                figwidth    = 0.99\linewidth]{todonotes}

\usepackage[all]{xy}
\usepackage{microtype}

\RequirePackage{enumitem}

\numberwithin{equation}{section}
\setlist[itemize]{font = \upshape, before = \leavevmode}
\setlist[enumerate]{font = \upshape, before = \leavevmode}
\setlist[description]{before = \leavevmode}

\usepackage[backend = biber, style = alphabetic, url = false, isbn = false, doi = false,giveninits = true]{biblatex}
\DeclareFieldFormat*{title}{#1}

\everymath
{
    \ifodd\value{page}          
        \allowdisplaybreaks[1]  
        \else                   
        \allowdisplaybreaks[4]  
    \fi                         
}

\theoremstyle{plain}
\newtheorem{theorem}{Theorem}[section]
\newtheorem{lemma}[theorem]{Lemma}

\newtheorem{prop}[theorem]{Proposition}

\theoremstyle{definition}
\newtheorem{definition}[theorem]{Definition}

\newtheorem{example}[theorem]{Example}

\theoremstyle{remark}
\newtheorem{remark}[theorem]{Remark}

\DeclareMathOperator{\Gal}{Gal}

\DeclareMathOperator{\SL}{SL}

\DeclareMathOperator{\coker}{coker}
\DeclareMathOperator{\Spec}{Spec}

\DeclareMathOperator{\Pl}{Pl}

\DeclareMathOperator{\rk}{rk}

\DeclarePairedDelimiter{\p}{\lparen}{\rparen}          
\DeclarePairedDelimiter{\ip}{\langle}{\rangle}         
\DeclarePairedDelimiter{\pfister}{\langle\!\langle}{\rangle\!\rangle}

\newcommand{\frakm}{\mathfrak{m}}

\newcommand{\WK}{\mathrm{WK}}

\newcommand{\rmc}{\mathrm{c}}

\newcommand{\rmH}{\mathrm{H}}

\newcommand{\rmI}{\mathrm{I}}

\newcommand{\rmK}{\mathrm{K}}

\newcommand{\rmM}{\mathrm{M}}
\newcommand{\rmN}{\mathrm{N}}

\newcommand{\rmW}{\mathrm{W}}

\newcommand{\rmr}{\mathrm{r}}

\newcommand{\sspt}{\mathbf{1}}
\newcommand{\A}{\mathbf{A}}

\newcommand{\C}{\mathbf{C}}

\newcommand{\F}{\mathbf{F}}

\newcommand{\J}{\mathbf{J}}

\newcommand{\Q}{\mathbf{Q}}
\newcommand{\R}{\mathbf{R}}

\newcommand{\Z}{\mathbf{Z}}

\newcommand{\calO}{\mathcal{O}}

\newcommand{\id}{\mathrm{id}}

\newcommand{\MW}{\mathrm{MW}}
\newcommand{\rmMW}{\mathrm{MW}}

\newcommand{\ord}{\mathrm{ord}}
\newcommand{\sgn}{\mathrm{sgn}}

\newcommand{\wt}{\widetilde}

\newcommand{\ol}{\overline}

\newcommand{\defeq}{\vcentcolon=}

\newcommand{\nc}{\mathrm{nc}}

\newcommand{\GW}{\mathrm{GW}}
\newcommand{\vol}{\mathrm{vol}}

\address{Department of Mathematics, University of Oslo, Norway}
\email{\href{mailto:hakon.kolderup@hotmail.com}{hakon.kolderup@hotmail.com}}



\begin{document}

\title[Number theoretic aspects of Milnor--Witt $\rmK$-theory]{Remarks on classical number theoretic aspects\\ of Milnor--Witt K-theory}
\date{}
\author{Håkon Kolderup}
\maketitle

\begin{abstract}
We record a few observations on number theoretic aspects of Milnor--Witt $\rmK$-theory, focusing on generalizing classical results on reciprocity laws, Hasse's norm theorem and $\rmK_2$ of number fields and rings of integers. 
\end{abstract}


\section{Introduction}

The algebraic $\rmK$-groups of number fields and rings of integers are known to encode deep arithmetic information. This is witnessed already by the computation of the zeroth and first $\rmK$-groups of the ring of integers $\calO_F$ in a number field $F$: indeed, the torsion subgroup of $\rmK_0(\calO_F)$ is precisely the ideal class group of $F$, while $\rmK_1(\calO_F)$ is the group of units in $\calO_F$. In the 70's, Tate discovered that the second $\rmK$-group of $F$ is inherently related to reciprocity laws on $F$ \cite{Tate-symbols}. More precisely, Tate found that in the case when $F=\Q$ we have
\[
\rmK_2(\Q)\cong\Z/2\oplus\bigoplus_{p\text{ }\rm{prime}}\F_p^\times.
\]
Tate's proof method follows essentially Gauss' first proof of the quadratic reciprocity law involving an induction over the primes. One can show that Tate's structure theorem for $\rmK_2(\Q)$ gives rise to the product formula for Hilbert symbols over $\Q$, which is an equivalent formulation of the law of quadratic reciprocity. See \cite[II §7]{Gras} for details.

Besides the algebraic $\rmK$-groups there are several other important invariants attached to a number field $F$. A noteworthy example is the Witt ring $\rmW(F)$ of $F$, which subsumes much of the theory of quadratic forms over $F$. For instance, the celebrated Hasse--Minkowski's local-global principle can be formulated in terms of the Witt ring by stating that an element of $\rmW(\Q)$ is trivial if and only if it maps to zero in $\rmW(\R)$ and in $\rmW(\Q_p)$ for each prime $p$ \cite[IV Corollary 2.4]{Milnor-Husemoller}. 
Another example is given by the \emph{Milnor $\rmK$-groups of $F$}, $\rmK_n^\rmM(F)$, introduced by Milnor in his 1970 paper \cite{Milnor-quad-forms}. By definition, the Milnor $\rmK$-groups of $F$ coincide with the algebraic $\rmK$-groups of $F$ in degrees $0$ and $1$, while Matsumoto's theorem on $\rmK_2$ of fields \cite{Matsumoto-thm} implies that also $\rmK_2^\rmM(F)\cong\rmK_2(F)$. In higher degrees, however, these groups are in general different.
On the other hand, Milnor proved in \cite{Milnor-quad-forms} that the Milnor $\rmK$-groups of $F$ are intimately linked with the Witt ring of $F$. 
The understanding of this connection between Milnor $\rmK$-theory and quadratic forms was greatly enhanced in the wake of Morel and Voevodsky's introduction of motivic homotopy theory \cite{Morel-Voevodsky}, and in particular by Orlov, Vishik and Voevodsky's solution of Milnor's conjecture on quadratic forms \cite{Orlov-Vishik-Voevodsky}. In fact, both the Milnor $\rmK$-groups and the Witt ring was set in new light in the context of motivic homotopy groups. More precisely, Hopkins and Morel introduced the so-called \emph{Milnor--Witt $\rmK$-groups} $\rmK_*^\MW(F)$ of $F$, and Morel showed in \cite[Theorem 6.4.1]{Morel-sphere-spt} that for any integer $n$, there is a canonical isomorphism
\begin{align}
\pi_{n,n}\sspt\cong \rmK^{\rmMW}_{-n}(F).
\end{align}
Here $\pi_{n,n}\sspt$ denotes the motivic homotopy group of the motivic sphere spectrum $\sspt$ over $F$ in bidegree $(n,n)$. We refer the reader to, e.g., \cite{mot-htpy-gr-survey} for a survey on motivic homotopy groups. The Milnor--Witt $\rmK$-groups of $F$ are equipped with forgetful maps to the Milnor $\rmK$-groups as well as to the Witt ring, and can therefore be considered as an enhancement of the Milnor $\rmK$-groups of $F$ which also takes into account information coming from quadratic forms defined over $F$. 

Below we investigate what number theoretic information the lower Milnor--Witt $\rmK$-groups carry. In particular, we define Hilbert symbols and idèle class groups in this setting; we consider the connection between $\rmK_2^\MW(F)$ and reciprocity laws; we compute $\rmK_2^\MW(\calO_F)$ in a few explicit examples; and we show a Hasse type norm theorem for $\rmK_2^\MW$. Since the Milnor--Witt $\rmK$-groups contain information coming from quadratic forms, we obtain analogs of classical results that are more sensitive to the infinite real places of the number field than the ordinary $\rmK$-groups.

\subsection{Outline}
In \Cref{section:prelims} we start by recalling the definition and basic properties of Milnor--Witt $\rmK$-theory, before we move on to computing Milnor--Witt $\rmK$-groups of the rationals as well as some local fields. We finish this preliminary section by defining valuations in the setting of Milnor--Witt $\rmK$-theory, lifting the classical valuations on a number field.

In \Cref{section:topology} we put a topology on the first Milnor--Witt $\rmK$-group $\rmK_1^\MW(F_v)$ of a completion of the number field $F$, in such a way that $\rmK_1^\MW(F_v)$ becomes a covering of $F_v^\times$. This is used in \Cref{section:ideles} where we define idèles and idèle class groups in the setting of Milnor--Witt $\rmK$-theory. We show that the associated volume zero idèle class group $\wt C^0_F$ is again a compact topological group extending the classical compact group $C_F^0$. See \Cref{prop:vol-zero} for more details.

In \Cref{section:Moore} we shift focus from $\rmK_1^\MW$ to $\rmK_2^\MW$. We define Hilbert symbols on Milnor--Witt $\rmK$-groups and show a Moore reciprocity sequence in this setting; see \Cref{prop:MWmoore}. We then move on to the study of $\rmK_2^\MW$ of rings of integers in \Cref{section:K2}. Finally, in \Cref{section:Hasse} we show an analog of Hasse's norm theorem similar to the generalizations of Bak--Rehmann \cite{Bak-Rehmann} and Østvær \cite{PA-Hasse}.

\subsection{Conventions and notation}Throughout we let $F$ denote a number field of signature $(r_1,r_2)$, and we let $\Pl_F$ denote the set of places of $F$. For any $v\in\Pl_F$, we let $F_v$ denote the completion of $F$ at the place $v$, and we let $i_v\colon F\hookrightarrow F_v$ denote the embedding of $F$ into $F_v$. By Ostrowski's theorem, $\Pl_F$ decomposes as a disjoint union $\Pl_F=\Pl_0\cup\Pl_\infty$ of the finite and infinite places of $F$, respectively. The set $\Pl_\infty$ of infinite places of $F$ decomposes further into the sets $\Pl_\infty^{\rmr}$ and $\Pl_\infty^\rmc$ of real and complex infinite places, respectively. Finally, let $\Pl_F^\nc$ denote the set of noncomplex places of $F$.

In order to streamline the notation with the literature, we will often use the notations $\rmK_n^\rmM(F)$ and $\rmK_n(F)$ interchangeably whenever $n\in\{0,1,2\}$.

\subsection{Acknowledgments}
I thank Ambrus Pál for interesting discussions and for encouraging me to write this text. Furthermore, I am very grateful to Jean Fasel for particularly helpful comments on a draft, and to Kevin Hutchinson for his interest and for useful comments. I would also like to thank the anonymous referee for a careful reading and for many valuable remarks that helped improve the exposition.
This work was supported by the RCN Frontier Research Group Project no. 250399.

\section{Preliminaries}\label{section:prelims}
\subsection{Milnor--Witt \texorpdfstring{$\rmK$}{K}-theory}\label{section:KMW}
We start out by providing some generalities on Milnor--Witt $\rmK$-groups, mostly following Morel's book \cite{Morel-over-a-field}.

\begin{definition}[Hopkins--Morel]
The \emph{Milnor--Witt $\rmK$-theory} $\rmK_*^{\rmMW}(F)$ of $F$ is the graded associative $\Z$-algebra with one generator $[a]$ of degree $+1$ for each unit $a\in F^\times$, and one generator $\eta$ of degree $-1$, subject to the following relations:
\begin{center}
\begin{tabular}{ll}
(I) $[a][1-a]=0$ for any $a\in F^\times\setminus\{1\}$ & (Steinberg relation).\\
(II) $[ab]=[a]+[b]+\eta[a][b]$ & (twisted $\eta$-logarithmic relation).\\
(III) $\eta[a]=[a]\eta$ & ($\eta$-commutativity).\\
(IV) $(2+\eta[-1])\eta=0$ & (hyperbolic relation).
\end{tabular}
\end{center}
We let $\rmK_n^{\rmMW}(F)$ denote the $n$-th graded piece of $\rmK_*^{\rmMW}(F)$. The product $[a_1]\cdots[a_n]\in \rmK_n^{\rmMW}(F)$ may also be denoted by $[a_1,\dots,a_n]$; by \cite[Lemma 3.6 (1)]{Morel-over-a-field}, these symbols generate the abelian group $\rmK_n^\MW(F)$.
\end{definition}

\subsubsection{Milnor--Witt $\rmK$-theory and quadratic forms}
Let us explain the relationship between Milnor--Witt $\rmK$-theory and quadratic forms. Recall for instance from \cite{Milnor-Husemoller,Scharlau} that a \emph{symmetric bilinear form}\footnote{The connection between symmetric bilinear forms and quadratic forms is as follows. Any symmetric bilinear form $\beta$ gives rise to a quadratic form $q$ by setting $q(x)\defeq\beta(x,x)$. Moreover, if $F$ is of characteristic different from $2$, then any quadratic form over $F$ arises uniquely in this way from a symmetric bilinear form over $F$.} over $F$ is a finite dimensional $F$-vector space $V$ together with a nondegenerate symmetric bilinear map $\beta\colon V\times V\to F$. The group completion of the semiring of isomorphism classes of symmetric bilinear forms over $F$ is called the \emph{Grothendieck--Witt ring} of $F$, denoted $\GW(F)$. Any unit $u$ of $F$ defines a symmetric bilinear form $\ip{u}$ whose underlying vector space is just $F$, and whose bilinear map $\beta$ is given by $\beta(x,y)=uxy$. In fact, the forms $\ip{u}$ for $u\in F^\times$ additively generate the Grothendieck--Witt ring of $F$ \cite[Lemma 3.9]{Morel-over-a-field}. Sending the element $1+\eta[u]\in\rmK_0^\MW(F)$ to $\ip{u}\in\GW(F)$ gives a well defined ring homomorphism from $\rmK_0^\MW(F)$ to $\GW(F)$ which is in fact an isomorphism \cite[Chapter 3]{Morel-over-a-field}. In light of this isomorphism we will, for any $u\in F^\times$, denote the element $1+\eta[u]\in\rmK_0^\MW(F)$ also by $\ip{u}$.

The addition and multiplication in the ring $\GW(F)$ stems from direct sum and tensor product of vector spaces over $F$. The form $H\defeq\ip1+\ip{-1}\in\GW(F)$ is called the \emph{hyperbolic plane}, and this form generates an ideal which is isomorphic to $\Z$. The resulting quotient ring $\rmW(F)\defeq\GW(F)/(H)$ is called the \emph{Witt ring} of $F$. In the defining relation (IV) of Milnor--Witt $\rmK$-theory above, the element $2+\eta[-1]=1+\ip{-1}\in\rmK_0^{\MW}(F)$ corresponds to $H\in\GW(F)$. Thus the hyperbolic relation essentially implies that multiplication by $\eta$ on the negative Milnor--Witt $\rmK$-groups becomes an isomorphism, identifying $\rmK_{n}^\MW(F)$ with $\rmW(F)$ for all $n<0$; see \cite[Lemma 3.10]{Morel-over-a-field} for details.

Taking the rank of forms over $F$ defines a ring homomorphism $\GW(F)\to\Z$, which descends to a homomorphism $\rmW(F)\to\Z/2$. The kernel of the map $\rmW(F)\to\Z/2$ consists of the even-dimensional forms in the Witt ring, and is referred to as the \emph{fundamental ideal} of $F$, denoted $\rmI(F)$. The powers $\rmI^n(F)$ of the fundamental ideal are generated by the so-called \emph{Pfister forms} $\pfister{a_1,\dots,a_n}\defeq(\ip1-\ip{a_1})\cdots(\ip1-\ip{a_n})$. For $n\ge1$ there is a group homomorphism from $\rmK_n^\MW(F)$ to $\rmI^n(F)$ given by mapping $[a_1,\dots,a_n]$ to the Pfister form $\pfister{a_1,\dots,a_n}$. We can use this map to define, for each infinite real place $v$ of $F$, a \emph{signature homomorphism} $\wt\sgn_v\colon\rmK_n^\MW(F)\to\Z$ as the composition
\begin{equation}\label{eq:signature-hom}
\wt\sgn_v\colon\rmK_n^\MW(F)\to\rmK_n^\MW(F_v)\to\rmI^{n}(F_v)\xrightarrow{\cong}\Z.
\end{equation}
Here the last homomorphism is given by the signature of quadratic forms \cite[p. 62]{Milnor-Husemoller}: for example, for $a\in F_v^\times$, the signature of $\ip a$ is $+1$ if $a$ is positive, and $-1$ otherwise. Thus the signature of $\pfister{-1,\dots,-1}\in\rmI^n(F_v)$ is $2^n$, and this defines an isomorphism $\rmI^{n}(F_v)\cong\Z$ by \cite[III Corollary 2.7]{Milnor-Husemoller}, carrying the generator $\pfister{-1,\dots,-1}$ to $1\in\Z$.

As the name suggests, Milnor--Witt $\rmK$-theory is also related to Milnor $\rmK$-theory. Indeed, for any $n\ge0$ there is a surjective homomorphism $p\colon \rmK_n^\MW(F)\to\rmK_n^\rmM(F)$ determined by killing $\eta$ and sending $[a]$ to $\{a\}\in\rmK_1^\rmM(F)$. Its kernel is $\rmI^{n+1}(F)$, the $(n+1)$-th power of the fundamental ideal in the Witt ring of $F$. The above discussion is subsumed by the following pullback square, which is proved in \cite[Theorem 5.3]{Morel-Hopkins}:
\begin{equation}\label{eq:KMW-pullback-sq}
\begin{tikzcd}
\rmK_n^\MW(F)\ar{d}\ar{r}{p} & \rmK_n^\rmM(F)\ar{d}\\
\rmI^n(F)\ar{r} & \rmI^n(F)/\rmI^{n+1}(F)
\end{tikzcd}
\end{equation}

\subsubsection{The residue map}\label{section:res-and-transfer}
In Milnor $\rmK$-theory, there is a residue map, or \emph{tame symbol}, $\partial_v\colon\rmK_*^\rmM(F)\to\rmK_{*-1}^\rmM(k(v))$ defined for each finite place $v$ of $F$. These homomorphisms assemble to a \emph{total residue map}
\[
\partial\colon \rmK_*^\rmM(F)\to\bigoplus_{v\in\Pl_0}\rmK_{*-1}^\rmM(k(v))
\]
given as $\partial=\bigoplus_{v\in\Pl_0}\partial_v$. Morel shows in \cite[Theorem 3.15]{Morel-over-a-field} that the same is true for Milnor--Witt $\rmK$-theory. More precisely, for each uniformizer $\pi_v$ for $v$, there is a unique homomorphism $\partial_v^{\pi_v}$ giving rise to a graded homomorphism
\[
\partial=\bigoplus_{v\in\Pl_0}\partial_v^{\pi_v}\colon\rmK_*^\MW(F)\to\bigoplus_{v\in\Pl_0}\rmK_{*-1}^\MW(k(v))
\]
which commutes with $\eta$ and satisfies $\partial_v^{\pi_v}([\pi_v,u_1,\dots,u_n])=[\ol u_1,\dots,\ol u_n]$ whenever the $u_i$'s are units modulo $\pi_v$. In contrast to the case for Milnor $\rmK$-theory, the maps $\partial_v^{\pi_v}$ depend on the choice of uniformizer $\pi_v$; this stems from the relation $[u\pi_v]=[u]+[\pi_v]+\eta[u,\pi_v]$. One can however define a twisted version of Milnor--Witt $\rmK$-theory in order to make the maps $\partial_v^{\pi_v}$ canonical. Indeed, for any field $k$ and any one-dimensional $k$-vector space $V$, let $\rmK_*^\MW(k,V)\defeq\rmK_*^\MW(k)\otimes_{\Z[k^\times]}\Z[V\setminus\{0\}]$, where $\Z[k^\times]$ acts by $u\mapsto\ip u$ on $\rmK_*^\MW(k)$ and by multiplication on $\Z[V\setminus\{0\}]$. Then the map $\partial_v\colon\rmK_*^\MW(F)\to\rmK_{*-1}^\MW(k(v),(\frakm_v/\frakm_v^2)^\vee)$ given by $\partial_v([\pi_v,u_1,\dots,u_n])=[\ol u_1,\dots,\ol u_n]\otimes\ol\pi_v$ is independent of the choice of uniformizer \cite[Remark 3.21]{Morel-over-a-field}.

\subsubsection{A few exact sequences}\label{section:on-seses}
Let us collect some short exact sequences involving Milnor--Witt $\rmK$-groups that will be used later in the text. First of all, the square \eqref{eq:KMW-pullback-sq} gives a short exact sequence
\begin{equation}\label{eq:fundamental-KMW-sequence}
0\to\rmI^{n+1}(F)\to\rmK_n^\MW(F)\xrightarrow{p}\rmK_n^\rmM(F)\to0.
\end{equation}
Here the map $\rmI^{n+1}(F)\to\rmK_n^\MW(F)$ is defined by sending the Pfister form $\pfister{a_1,\dots,a_{n+1}}$ to $\eta[a_1\dots,a_{n+1}]$.
There is a similar sequence with the fundamental ideal on the right hand-side; it takes the form \cite[p. 6]{Hutchinson-Tao2}
\begin{align}\label{eq:ses-with-witt-on-right}
0\to2\rmK_n^\rmM(F)\to\rmK_n^\MW(F)\to\rmI^n(F)\to0.
\end{align}
The left hand-side homomorphism is here given by mapping $2\{a_1,\dots,a_n\}$ to $h[a_1,\dots,a_n]$, where $h\defeq 1+\ip{-1}$ is the hyperbolic plane.

On the other hand, there is a fundamental computation by Morel  \cite[Theorem 3.24]{Morel-over-a-field} (which follows Milnor's computation in the case of Milnor $\rmK$-theory \cite{Milnor}) showing that there is a split short exact sequence
\begin{align}\label{eq:Morel-split-ses}
0\to \rmK_{n}^\MW(F)\to\rmK_{n}^\MW(F(t))\xrightarrow{\partial}\bigoplus_{x\in(\A^1_F)^{(1)}}\rmK_{n-1}^\MW(k(x))\to0.
\end{align}
Here $x$ runs over all closed points of $\A^1_F$.

\subsubsection{Transfers in Milnor--Witt \texorpdfstring{$\rmK$}{K}-theory}\label{transfer-map}
If $L/F$ is an extension of number fields, recall that there exist \emph{norm maps}, or \emph{transfer maps} in Milnor $\rmK$-theory \cite{KatoII},
\[
\tau_{L/F}\colon\rmK_n^\rmM(L)\to\rmK_n^\rmM(F),
\]
which generalize the classical norm $\rmN_{L/F}\colon L^\times\to F^\times$. Similar maps exist also for Milnor--Witt $\rmK$-theory, and are constructed in the same way as for the Milnor $\rmK$-groups. We briefly recall this construction. 
Let $\alpha$ be a primitive element for the extension of number fields $L/F$, so that $L=F(\alpha)$. Let $P\in F[t]$ be the minimal polynomial of $\alpha$. The transfer map 
\[\tau_{L/F}\colon\rmK_n^\MW(L)\to\rmK_n^\MW(F)
\] 
is defined by using the split short exact sequence \eqref{eq:Morel-split-ses}
as follows. Choose a section $s\defeq\bigoplus_xs_x$ of $\partial$, and let $y\in(\A^1_F)^{(1)}$ be the closed point corresponding to $P$. Then we define $\tau_{L/F}$ as the composition
\[
\tau_{L/F}\colon\rmK_n^\MW(L)\xrightarrow{\cong}\rmK_n^\MW(k(y))\xrightarrow{s_y}\rmK_{n+1}^\MW(F(t))\xrightarrow{-\partial_\infty^{-1/t}}\rmK_n^\MW(F),
\]
where $\partial_\infty^{-1/t}$ is the residue map corresponding to the valuation on $F(t)$ with uniformizer $-1/t$. It is a difficult theorem, proved by Morel in \cite[Theorem 4.27]{Morel-over-a-field}, that the map $\tau_{L/F}$ does not depend on the choices made.

\subsection{Milnor--Witt \texorpdfstring{$\rmK$}{K}-theory of finite and local fields}
We compute some Milnor--Witt $\rmK$-groups of finite and local fields. The results follow readily from the structure of the corresponding Milnor $\rmK$-groups along with some knowledge about fundamental ideals.

\begin{prop}\label{prop:K2MW(Fv)}
\begin{enumerate}
\item[(i)] If $\F$ is a finite field, then $\rmK_n^\MW(\F)\cong\rmK_n^\rmM(\F)$ for all $n\ge1$.
\item[(ii)] Let $v\in\Pl_F$ be a place of $F$. We have isomorphisms of abelian groups
\[
\rmK_n^{\MW}(F_v)\cong \begin{cases}
\rmK_n^{\rmM}(F_v), & v\in\Pl_0,\quad \hspace{4pt}n\ge2\\
\Z\oplus A_v, & v\in\Pl_\infty^{\rmr},\quad n\ge1\\
\rmK_n^\rmM(F_v), & v\in\Pl_\infty^{\mathrm c},\quad n\ge0,
\end{cases}
\]
where the $A_v$'s are uniquely divisible abelian groups.
\end{enumerate}
\end{prop}

\begin{remark}\label{remark:K2M-local-field-w-references}
The Milnor $\rmK$-groups of local fields are known: if $v$ is a finite place of $F$, then $\rmK_2^\rmM(F_v)\cong\mu(F_v)\oplus A_v$, where $A_v$ is a uniquely divisible group \cite{Merkurjev-torsion-K2}, while $\rmK_n^\rmM(F_v)$ is uniquely divisible for $n\ge3$ \cite{Sivitskii}. If $v$ is an infinite real place of $F$, then for all $n\ge1$, $\rmK_n^\rmM(F_v)$ is the direct sum of a cyclic group of order $2$ generated by $\{-1,\dots,-1\}$ and a uniquely divisible group \cite[Example 7.2]{WeibelK}. Finally, if $v$ is a complex place of $F$ then $\rmK_n^\rmM(F_v)$ is uniquely divisible for all $n\ge1$ \cite[Example 7.2]{WeibelK}.
\end{remark}

\begin{proof}[Proof of \Cref{prop:K2MW(Fv)}]
The first claim follows from the exact sequence \eqref{eq:fundamental-KMW-sequence} since $\rmI^2(\F)=0$ for any finite field $\F$ \cite[p. 81]{Milnor-Husemoller}.

For (ii), first assume $v\in\Pl_\infty^{\mathrm c}$. Then the statement follows from the fact that $\C$ is quadratically closed \cite[Proposition 3.13]{Morel-over-a-field}. If $v\in\Pl_0$ we have $\rmI^3(F_v)=0$ by \cite[p. 81]{Milnor-Husemoller}, and hence the exact sequence \eqref{eq:fundamental-KMW-sequence}
yields $\rmK_n^{\MW}(F_v)\cong \rmK_n^\rmM(F_v)$ for each $n\ge2$. 
Finally, suppose $v\in\Pl_\infty^\rmr$. Then we have $\rmI^{n+1}(F_v)\cong\Z$, generated by the Pfister form $\pfister{-1,\dots,-1}$ \cite[p. 81]{Milnor-Husemoller}. Furthermore, by \Cref{remark:K2M-local-field-w-references}, $\rmK_n^\rmM(F_v)\cong\Z/2\oplus A$, where $A$ is a uniquely divisible abelian group. Using that $2\rmK_n^\rmM(F_v)\cong A$, the sequence \eqref{eq:ses-with-witt-on-right} above reduces in this case to
\[
0\to A\to\rmK_n^\MW(F_v)\to\Z\to0.
\]
The right hand-side being free, this sequence splits and the result follows.
\end{proof}

\begin{remark}\label{rmk:isos}
In the case of finite places and $n=2$, the isomorphisms appearing in \Cref{prop:K2MW(Fv)} are given by the classical local Hilbert symbols $(-,-)_v$ \cite[V §3]{Neukirch}. On the other hand, if $v\in\Pl_\infty^{\rmr}$, we can think of the signature map $\rmK_2^\MW(\R)\to\Z$ as a ``$\Z$-valued Hilbert symbol'' extending the classical $\Z/2$-valued Hilbert symbol on $\R$. We will return to this point of view in \Cref{section:Hilbert}.
\end{remark}

\subsection{Milnor--Witt \texorpdfstring{$\rmK$}{K}-theory of the rationals} 
The Witt- and Grothendieck--Witt ring of $\Q$ is determined for instance in \cite[IV §2]{Milnor-Husemoller}. Thus we know $\rmK_n^\MW(\Q)$ for $n\le0$. The the remaining groups are given as follows:

\begin{prop}\label{prop:MW(Q)}
For each $n\ge1$, the residue map $\partial$ defined in \Cref{section:res-and-transfer} induces an isomorphism of abelian groups
\[
\rmK_n^{\MW}(\Q)\cong\Z\oplus\bigoplus_{p\text{ }\rm{prime}}\rmK_{n-1}^{\MW}(\F_p).
\]
In particular, 
\[
\rmK_1^\MW(\Q)\cong\Z\oplus\bigoplus_{p\text{ }\rm{prime}}\GW(\F_p);\quad\rmK_2^\MW(\Q)\cong\Z\oplus\bigoplus_{p\text{ }\rm{prime}}\F_p^\times,
\]
while $\rmK_n^\MW(\Q)\cong\Z$ for $n\ge3$.
\end{prop}

\begin{proof}
For each $n\ge1$ let $\Lambda_n$ denote the kernel of the signature homomorphism $\wt\sgn\colon\rmK_n^\MW(\Q)\to\Z$ defined in \eqref{eq:signature-hom}. Since the target of $\wt\sgn$ is free, we have
$
\rmK_n^\MW(\Q)\cong\Z\oplus\Lambda_n,
$
and it remains to identify $\Lambda_n$.
Since $\rmK_n^\rmM(\Q)\cong\Z/2\oplus\bigoplus_{p\text{ prime}}\rmK_{n-1}^\rmM(\F_p)$ for $n\ge1$ \cite{AKII}, then using the sequence \eqref{eq:fundamental-KMW-sequence} it follows that we have a commutative diagram with exact rows
\[\begin{tikzcd}\label{eq:big-KMWQ-diagram}
0\ar{r} & \ker(\sgn)\ar{d}\ar{r} & \Lambda_n\ar{d}\ar{r} & \displaystyle{\bigoplus_{p\text{ prime}}}\rmK_{n-1}^\rmM(\F_p)\ar{d}\ar{r} & 0\\
 0\ar{r} & \rmI^{n+1}(\Q)\ar{d}{\sgn}\ar{r} & \rmK_n^\MW(\Q)\ar{d}{\wt\sgn}\ar{r}{p} & \Z/2\oplus\displaystyle{\bigoplus_{p\text{ prime}}}\rmK_{n-1}^\rmM(\F_p)\ar{r}\ar{d} & 0\\
 0 \ar{r} & \Z\ar{r} & \Z\ar{r} & \Z/2\ar{r} & 0.
\end{tikzcd}
\]
Here $\sgn$ denotes the signature map on quadratic forms \cite[p. 62]{Milnor-Husemoller}. From knowledge about the fundamental ideal of $\Q$ \cite[IV §2]{Milnor-Husemoller} we conclude the following:
\begin{itemize}
\item For $n=1$ the upper short exact sequence reads
\[
0\to\bigoplus_{p>2\text{ prime}}\Z/2\to\Lambda_1\to\bigoplus_{p\text{ prime}}\Z\to0.
\]
The right hand-side being free, this sequence splits, and we conclude that 
\[
\Lambda_1\cong\Z\oplus\bigoplus_{p>2\text{ prime}}(\Z\oplus\Z/2)\cong\bigoplus_{p\text{ prime}}\GW(\F_p).
\]
\item If $n\ge2$, then the map $\sgn$ is an isomorphism and the claim follows.
\end{itemize}
This finishes the proof.
\end{proof}

\subsection{Valuations}\label{subsect:val}

Let $v$ be a finite place of $F$. Classically, the $v$-adic discrete valuation on the local field $F_v$ is a homomorphism
$
\ord_v\colon F_v^\times\to\Z
$
which, by definition, coincides with the residue map $\partial_v\colon \rmK_1^\rmM(F_v)\to\rmK_0^\rmM(k(v))=\Z$ on Milnor $\rmK$-theory. The units $\calO_v^\times$ of the corresponding valuation ring $\calO_v$ then coincides with the kernel of $\ord_v$. Following \cite[I Definition 1.5]{Gras} we extend this picture to the real and complex places of $F$. Thus,  
if $v$ is an infinite real place of $F$ we denote by $\ord_v$ the homomorphism
$
F_v^\times\to\Z/2
$
under which $x\in F_v^\times\cong \R^\times$ maps to $0$ if $x>0$, and $1$ if $x<0$. On the other hand, we set $\ord_v\defeq0$ whenever $v$ is a complex infinite place. In any case, we let $\calO_v^\times$ denote the kernel of $\ord_v$.

\begin{definition}\label{def:valuations}
Let $v$ be a place of $F$. 
\begin{enumerate}
\item If $v\in\Pl_0$, we define
\[
\wt\ord_v\colon \rmK_1^\MW(F_v)\to\GW(k(v),(\frakm_v/\frakm_v^2)^\vee)
\]
by $\wt\ord_v\defeq\partial_v$, where $\partial_v$ is the residue map on Milnor--Witt $\rmK$-theory.
\item If $v\in\Pl_\infty^\rmr$, we define the homomorphism
\[
\wt\ord_v\colon\rmK_1^\MW(F_v)\to\Z
\]
as the signature homomorphism \eqref{eq:signature-hom}.
\item For $v\in\Pl_\infty^\rmc$, we let $\wt\ord_v$ be the trivial homomorphism on $\rmK_1^\MW(F_v)$.
\end{enumerate}
In any case, we let
$
\rmK_1^\MW(\calO_v)
$
denote the kernel of $\wt\ord_v$.
\end{definition}

\begin{remark}
For the infinite real places of $F$ we are in \Cref{def:valuations} not really taking $\rmK_1^\MW$ of a field or ring: the notation $\rmK_1^\MW(\calO_v)$ is only suggestive in order to obtain analogs of the unit groups attached to each place of $F$ \cite[I §1]{Gras}. 
For a finite place or complex infinite place $v$ of $F$, however, the above definition is the same as Morel's definition of Milnor--Witt $\rmK$-groups of valuation rings \cite[§3]{Morel-over-a-field}. Morel shows in \cite[Theorem 3.22]{Morel-over-a-field} that with this definition, the Milnor--Witt $\rmK$-theory of $\calO_v$ is generated as a ring by $\eta$ along with the symbols $[u]$ for $u$ a unit of $\calO_v$. It follows that this definition coincides with other definitions of Milnor--Witt $\rmK$-theory of rings given in the literature, for example that of Schlichting in \cite[§4]{Schlichting-euler}.
\end{remark}


\begin{lemma}\label{lemma:K1MWOv}
Let $v\in\Pl_F$ be a place of $F$.
\begin{itemize}
\item If $v$ is either an infinite place or a nondyadic finite place, then  $\rmK_1^\MW(\calO_v)\cong\calO_v^\times$ (where, by definition, $\calO_v^\times\defeq\R^\times_{>0}$ for $v\in\Pl_\infty^\rmr$ and $\calO_v^\times\defeq\C^\times$ for $v\in\Pl_\infty^\rmc$).
\item If $v$ is a dyadic place, then there is a short exact sequence
\[
0\to\rmI^2(F_v)\to\rmK_1^\MW(\calO_v)\to\calO_v^\times\to0.
\]
\end{itemize}
\end{lemma}

\begin{proof}
The statement is clear for the infinite places by \Cref{prop:K2MW(Fv)}. Let $v\in\Pl_0$ be a finite place of $F$, and consider the commutative diagram with exact rows
\[\begin{tikzcd}
0\ar{r} & \rmK_1^\MW(\calO_v)\ar{r}\ar{d}{p'} & \rmK_1^\MW(F_v)\ar{r}{\partial_v}\ar{d}{p} & \rmK_0^\MW(k(v))\ar{r}\ar{d}{p''} & 0\\
0\ar{r} & \calO_v^\times\ar{r} & F_v^\times\ar{r}{\ord_v} & \Z\ar{r} & 0. 
\end{tikzcd}\]
If $v\nmid2$, then the induced map $\ker p\cong \rmI^2(F_v)\cong\Z/2\to\ker p''\cong\rmI(k(v))\cong\Z/2$ is an isomorphism and so we conclude by the snake lemma.
If $v\mid2$, then $k(v)$ is quadratically closed and hence the map $p''\colon\rmK_0^\MW(k(v))\to\Z$ is an isomorphism. In this case the snake lemma applied to the above diagram yields $\ker p'\cong\ker p\cong\rmI^2(F_v)$.
\end{proof}

\section{Topology on \texorpdfstring{$\rmK_1^\MW(F_v)$}{K1MW}}\label{section:topology}

We now aim to put a topology on $\rmK_1^\MW(F_v)$, for each place $v$ of $F$, in such a way that the homomorphism $p\colon\rmK_1^\MW(F_v)\to F_v^\times$ becomes a covering map.

\subsubsection{}In general, suppose that we are given an abelian group $G$ (written additively) along with a subgroup $H$ of $G$ which is a topological group. We can then extend the topology on $H$ to a topology on $G$ as follows. Choose a set theoretical section $s$ of the quotient map $\pi\colon G\to G/H$, and let $G/H$ have the discrete topology. Then $s$ defines a partition
\[
G=\coprod_{x\in G/H}(s(x)+H)
\]
of $G$,
and there is a natural topology on each coset $(s(x)+H)\cong H$ coming from that on $H$. We can then declare a subset $U\subseteq G$ to be open if and only if $U\cap(s(x)+H)$ is open for every $x\in G/H$. This turns $G$  into a topological group.

\subsubsection{}
Now let $v$ be a place of $F$. Then, by taking $n=1$ in the short exact sequence \eqref{eq:ses-with-witt-on-right} we get the exact sequence
\[
0\to F_v^{\times2}\to\rmK_1^\MW(F_v)\to\rmI(F_v)\to0.
\]
In the notations above, we can then let $G\defeq\rmK_1^\MW(F_v)$ and $H\defeq F_v^{\times2}$.
The quotient map $\pi\colon G\to G/H$ is given as $[a]\mapsto\pfister{a}$. Using the set theoretical section $s(\pfister{x})\defeq[x]$ of $\pi$ along with the fact that $F_v^{\times2}$ is a topological group, we obtain by the above a topology on $\rmK_1^\MW(F_v)$. 

\begin{prop}\label{prop:top}For any $v\in\Pl_F$, the map $p\colon\rmK_1^\MW(F_v)\to F_v^\times$ is a covering of topological groups.
Furthermore, if $v$ is a finite place of $F$, then we have the following:
\begin{enumerate}
\item The map $p$ is proper.
\item The space $\rmK_1^\MW(F_v)$ is locally compact and totally disconnected, and $\rmK_1^\MW(\calO_v)$ is compact in $\rmK_1^\MW(F_v)$. 
\end{enumerate}
\end{prop}

\begin{proof}
If $v$ is a complex place, then by \Cref{prop:K2MW(Fv)} there is nothing to show. So we may assume that $v\in\Pl_F^\nc$. We have a commutative diagram with exact rows
\[
\begin{tikzcd}
0\ar{r} & F_v^{\times2}\ar{d}{=}\ar{r} & \rmK_1^\MW(F_v)\ar{d}{p}\ar{r} & \rmI(F_v)\ar{r}\ar{d} & 0\\
0\ar{r} & F_v^{\times2}\ar{r} & F_v^\times\ar{r} & \rmI(F_v)/\rmI^2(F_v)\ar{r} & 0.
\end{tikzcd}
\]
That $p$ is a covering map follows from this diagram along with the fact that the quotient topology on $F_v^\times/F_v^{\times2}\cong\rmI(F_v)/\rmI^2(F_v)$ is the discrete topology.

Now let $v$ be a finite place of $F$. Then $\rmI^2(F_v)\cong\Z/2$, so that we have an exact sequence
\[
0\to\Z/2\to\rmK_1^\MW(F_v)\xrightarrow{p} F_v^\times\to0.
\]
Thus $p$ is in this case a two sheeted covering map, hence proper.
The claim (2) follows from \Cref{lemma:K1MWOv} along with the properties of the topology on $F_v^\times$.
\end{proof}

\section{Idèles}\label{section:ideles}
Recall that the idèle group $\J_F$ of $F$ is the restricted product of the groups of units of the completions $F_v^\times$ with respect to the compact subgroups $\calO_v^\times$, with the restricted product topology. Equivalently, $\J_F$ can be defined as the direct limit $\varinjlim_S\J_F(S)$, where $S$ is a set of places of $F$ and
\[
\J_F(S)\defeq\prod_{v\in S}F_v^\times\times\prod_{v\not\in S}\calO_v^\times.
\]

\begin{definition}
For any finite set $S$ of places of $F$, put
\[
\wt\J_F(S)\defeq\prod_{v\in S}\rmK_1^\MW(F_v)\times\prod_{v\not\in S}\rmK_1^\MW(\calO_v).
\]
The \emph{Milnor--Witt idèle group} $\wt\J_F$ of $F$ is defined as the direct limit
\[
\wt\J_F\defeq\varinjlim_S\wt\J_F(S),
\]
where $S$ ranges over all finite subsets of $\Pl_F$. 
\end{definition}


\begin{prop}\label{prop:wtJses}
There is a short exact sequence
\[
0\to\bigoplus_{v\in\Pl_F}\rmI^2(F_v)\to\wt\J_F\xrightarrow{p}\J_F\to0,
\]
where the map $p$ is induced by the projection maps from Milnor--Witt $\rmK$-theory to Milnor $\rmK$-theory.
\end{prop}

\begin{proof}
For the definition of the homomorphism $p$, note that we have projection maps
$
\wt\J_F(S)\to\J_F(S)
$
for any finite set of places of $F$. Here $\J_F(S)\defeq\prod_{v\in S}F_v^\times\times\prod_{v\not\in S}\calO_v^\times$. The map $p$ is then the induced morphism on the colimit, which we notice is surjective.

Let $\Pl_2\defeq\{v\in\Pl_F:v\mid2\}$ denote the dyadic places of $F$. It follows from \Cref{lemma:K1MWOv} along with the sequence \eqref{eq:fundamental-KMW-sequence} that the kernel of the projection map $\wt\J_F(S)\to\J_F(S)$ is $\bigoplus_{v\in S\cup\Pl_2}\rmI^2(F_v)$. Passing to the colimit as $S$ varies, we thus find that the kernel of the map $\wt\J_F\to\J_F$ is $\bigoplus_{v\in\Pl_F}\rmI^2(F_v)$.
\end{proof}

\begin{definition}
Let $S$ be a finite set of places of $F$. We define a topology on $\wt\J_F(S)$ by taking the topology generated by the sets
\[
W_S\defeq\prod_{v\in S}U_v\times\prod_{v\not\in S}\rmK_1^\MW(\calO_v),
\]
where the $U_v$'s are open subsets of $\rmK_1^\MW(F_v)$ for each $v\in S$. This defines a topology on $\wt\J_F$ via the direct limit topology.
\end{definition}

\begin{lemma}
The Milnor--Witt idèle group $\wt\J_F$ is a locally compact topological group.
\end{lemma}

\begin{proof}
This follows from the fact that $\wt\J_F$ is the restricted product of the groups $\rmK_1^\MW(F_v)$ with respect to the subgroups $\rmK_1^\MW(\calO_v)$, which are compact for almost all $v$ by \Cref{prop:top}.
\end{proof}

\begin{definition}
Define the map
$
\wt i\colon\rmK_1^\MW(F)\to\wt\J_F
$
by $\wt i([x])\defeq([i_v(x)])_{v\in\Pl_F}$. 
We refer to the cokernel
$
\wt C_F\defeq\coker\wt i
$
as the \emph{Milnor--Witt idèle class group of $F$}.
\end{definition}

\begin{lemma}\label{lemma:wti-injective}
The map $\wt i$ is injective. Hence $\wt C_F=\wt\J_F/\wt i(\rmK_1^\MW(F))$.
\end{lemma}

\begin{proof}
By \Cref{prop:wtJses}, the kernel of the map $\wt\J_F\to\J_F$ is $\bigoplus_{v\in\Pl_F}\rmI^2(F_v)$. Hence there is a commutative diagram
\[\begin{tikzcd}
& & \rmI^2(F)\ar{r}\ar{d} & \displaystyle{\bigoplus_{v\in\Pl_F}}\rmI^2(F_v)\ar{d} & &\\
0\ar{r} & \ker\wt i\ar{r}\ar{d} & \rmK_1^\MW(F)\ar{r}{\wt i} \ar{d} & \wt\J_F\ar{r}\ar{d}{p} & \wt C_F\ar{r}\ar{d} & 0\\
& 0\ar{r} & F^\times\ar{r}{i} & \J_F\ar{r} &  C_F\ar{r} & 0.
\end{tikzcd}\]
By the Hasse--Minkowski theorem, the map $\rmI^2(F)\to\bigoplus_{v\in\Pl_F}\rmI^2(F_v)$ is injective. It follows from this and a diagram chase that $\ker\wt i=0$. 
\end{proof}

\subsubsection{Idèles of volume zero}
Recall that we have a volume map
$
\vol\colon\J_F\to\R_{>0}^\times
$
defined as $\vol(x)\defeq\prod_{v\in\Pl_F}|x_v|_{F_v}$ \cite[p. 361]{Neukirch}. Here $x=(x_v)_{v\in\Pl_F}\in\J_F$, and $|\cdot|_{F_v}$ is the $v$-adic absolute value on $F_v$. We let $\J_F^0$ denote the kernel of the volume map. By the classical product formula for absolute values \cite[Chapter III, Proposition 1.3]{Neukirch} we have $F^\times\subseteq\J_F^0$. The fact that the resulting quotient group $ C_F^0\defeq\J_F^0/F^\times$ is compact is an equivalent formulation of Dirichlet's unit theorem and the finiteness of the ideal class group \cite[I Proposition 4.2.7]{Gras}.

\begin{definition}\label{def:MW-volume-zero}
Let $p\colon\wt\J_F\to\J_F$ denote the projection map. We define a volume map $\wt\vol\colon\wt\J_F\to\R_{>0}^\times$ on $\wt\J_F$ by $\wt\vol\defeq\vol\circ p$, and put $\wt\J_F^0\defeq\ker(\wt\vol)$. Let $\wt C_F^0$ be the quotient group 
\[
\wt C_F^0\defeq\wt\J_F^0/\wt i(\rmK_1^\MW(F)).
\] 
\end{definition}

\begin{remark}
By \Cref{lemma:wti-injective} along with the product formula for absolute values we have $\wt i(\rmK_1^\MW(F))\subseteq\wt\J_F^0$, justifying the definition of $\wt C_F^0$.
\end{remark}

\begin{prop}\label{prop:vol-zero}
There is a short exact sequence
\[
0\to\Z/2\to\wt C_F^0\to C_F^0\to0.
\]
Hence $\wt C_F^0$ is a compact topological group.
\end{prop}

\begin{proof}
By definition, the kernel $\bigoplus_{v\in\Pl_F}\rmI^2(F_v)$ of the projection map $p\colon \wt\J_F\to\J_F$ is contained in $\wt\J_F^0$. Consider the commutative diagram with exact rows
\[\begin{tikzcd}
0\ar{r} & \rmK_1^\MW(F)\ar{r}{\wt i}\ar{d} & \wt\J_F^0\ar{r}\ar{d}{p} & \wt C_F^0\ar{r}\ar{d} & 0\\
0\ar{r} & F^\times\ar{r}{i} & \J_F^0\ar{r} &  C_F^0\ar{r} & 0.
\end{tikzcd}\]
By the snake lemma it suffices to show that the cokernel of the map $\rmI^2(F)\to\bigoplus_{v\in\Pl_F}\rmI^2(F_v)$ is $\Z/2$. But this follows from the snake lemma applied to the commutative diagram with exact rows
\[\begin{tikzcd}
0\ar{r} & \rmI^3(F)\ar{r}\ar{d}{\iota}[swap]{\cong} & \rmI^2(F)\ar{r}\ar{d} & \rmI^2(F)/\rmI^3(F)\ar{r}\ar{d} & 0\\
0\ar{r} & \displaystyle{\bigoplus_{v\in\Pl_F}\rmI^3(F_v)}\ar{r} & \displaystyle{\bigoplus_{v\in\Pl_F}\rmI^2(F_v)}\ar{r} & \displaystyle{\bigoplus_{v\in\Pl_F}\rmI^2(F_v)/\rmI^3(F_v)}\ar{r} & 0,
\end{tikzcd}\]
using the fact that the left hand vertical map $\iota$ is an isomorphism by \Cref{lemma:I3} below, and that the cokernel of the right hand vertical map is $\Z/2$ by \cite[Lemma A.1]{Milnor-quad-forms}.
\end{proof}

\begin{lemma}\label{lemma:I3}
For any number field $F$, the canonical map 
$
\iota\colon \rmI^3(F)\to\bigoplus_{v\in\Pl_F}\rmI^3(F_v)
$
induced by the embeddings $i_v\colon F\hookrightarrow F_v$ is an isomorphism.
\end{lemma}

\begin{proof}
The map is injective by the Hasse--Minkowski theorem. We must show that it is surjective.

If $v\in\Pl_0\cup\Pl_\infty^\rmc$, then $\rmI^3(F_v)=0$ and there is nothing to show. On the other hand, if $v$ is an infinite real place, then $\rmI^3(F_v)\cong\Z$ by \cite[III Corollary 2.7]{Milnor-Husemoller}. By strong approximation \cite[p. 193]{Neukirch} we can find an element $a\in F$ which is negative in the $i$-th ordering on $F$ and positive otherwise. Then $\pfister{-1,-1,a}\in \rmI^3(F)$ maps to the $i$-th unit vector in $\bigoplus_{v\in\Pl^{\rmr}} \rmI^3(F_v)$.
\end{proof}

\begin{remark}
One can speculate on whether there is a variant of the abelianized étale fundamental group of $\Spec(\calO_F)$ which is the recipient of a reciprocity map defined on the Milnor--Witt idèle class group $\wt C_F$ and which lifts the classical Artin reciprocity map $\rho\colon C_F\to\Gal(L/F)^{\mathrm{ab}}$. 

\end{remark}

\section{A Moore reciprocity sequence for Milnor--Witt \texorpdfstring{$\rmK$}{K}-theory}\label{section:Moore}

The classical result of Moore on uniqueness of reciprocity laws states that there is an exact sequence 
\[
0\to \WK_2(F)\to\rmK_2(F)\xrightarrow{h}\bigoplus_{v\in\Pl_F^\nc}\mu(F_v)\xrightarrow{\pi}\mu(F)\to0.
\]
Here $h$ denotes the global Hilbert symbol, and the group $\WK_2(F)$ is known as the \emph{wild kernel} \cite[II §7]{Gras}. Moreover, the map $\pi$ is defined as 
\[
\pi((\zeta_v)_v)\defeq\prod_{v\in\Pl_F^\nc}\zeta_v^{m_v/m},
\]
where $m_v\defeq\#\mu(F_v)$ and $m\defeq\#\mu(F)$.
In this section we will show that also $\rmK_2^\MW(F)$ fits into a similar exact sequence.

\subsection{Hilbert symbols}\label{section:Hilbert}
In order to obtain a Moore reciprocity sequence for $\rmK_2^\MW$, we first need to define Hilbert symbols in the setting of Milnor--Witt $\rmK$-theory. These should be particular instances of maps of the following type:

\begin{definition}
Let $A$ be a $\GW(F)$-module. A \emph{Milnor--Witt symbol on $F$ with values in $A$} is a $\GW(F)$-bilinear map
\[
(-,-)\colon\rmK_1^\MW(F)\times\rmK_1^\MW(F)\to A
\]
satisfying $([a],[1-a])=0$ for all $a\in F^\times\setminus\{1\}$.
\end{definition}

\begin{remark}
Note that any abelian group $A$ is also a $\GW(F)$-module via the rank map $\rk\colon\GW(F)\to \Z$. Thus we should think of the definition of a Milnor--Witt symbol as a lift of the classical notion of a \emph{symbol}, i.e., a $\Z$-bilinear map $(-,-)\colon F^\times\times F^\times\to A$ satisfying the Steinberg relation \cite{Tate-symbols}. Just as $\rmK_2(F)$ is the universal object with respect to symbol maps, it follows from \cite[Remark 3.2]{Morel-over-a-field} that $\rmK_2^\MW(F)$ is the universal object with respect to Milnor--Witt symbols.
\end{remark}

\begin{definition}\label{def:Bv}
For any place $v$ of $F$, let $B_v$ be the group 
\[
B_v\defeq\begin{cases}\mu(F_v), & v\in\Pl_0\\
\Z, & v\in\Pl_\infty^{\rmr}\\
0, & v\in\Pl_\infty^\rmc.\end{cases}
\]
Furthermore, define a map $q_v\colon B_v\to\mu(F_v)$ by letting $q_v$ be the identity if $v\in\Pl_0$; reduction modulo $2$ if $v\in\Pl_\infty^{\rmr}$; or the trivial homomorphism if $v\in\Pl_\infty^\rmc$. Finally, write $q\defeq\bigoplus_{v\in\Pl_F^\nc}q_v$.
\end{definition}

\subsubsection{}By \Cref{prop:K2MW(Fv)} we have $\rmK_2^\MW(F_v)\cong B_v\oplus A_v$ for each $v\in\Pl_F^\nc$. Thus we can define, for any noncomplex place $v$ of $F$, the homomorphism $b_v\colon\rmK_2^\MW(F_v)\to B_v$ as the composition of the isomorphism $\rmK_2^\MW(F_v)\cong B_v\oplus A_v$ followed by the projection $B_v\oplus A_v\to B_v$. Thus, if $v$ is a finite place of $F$ then $b_v$ is just the classical local Hilbert symbol, while for $v\in\Pl_\infty^\rmr$ the map $b_v$ is the signature homomorphism \eqref{eq:signature-hom}. 

\begin{definition}
For each $v\in\Pl_F^\nc$, let $h_v^{\MW}\colon\rmK_2^\MW(F)\to B_v$ denote the composite
\[
\rmK_2^{\MW}(F)\xrightarrow{i_v} \rmK_2^{\MW}(F_v)\xrightarrow{b_v} B_v,
\]
where the first map is induced by the embedding $i_v\colon F\hookrightarrow F_v$. Moreover, let 
\[h^{\MW}\colon \rmK_2^{\MW}(F)\to \bigoplus_{v\in\Pl_F^{\nc}} B_v\] 
be the map $h^{\MW}\defeq\bigoplus_{v\in\Pl_F^{\nc}}h_v^{\MW}$.
\end{definition}

\begin{definition}Let $v\in\Pl_v^\nc$ be a noncomplex place of $F$. We define the \emph{local Milnor--Witt Hilbert symbol at $v$}, denoted $(-,-)_v^{\MW}$, as the composite
\[
(-,-)_v^{\MW}\colon \rmK_1^{\MW}(F)\times \rmK_1^{\MW}(F)
\to \rmK_2^{\MW}(F)
\xrightarrow{h_v^{\MW}} B_v.
\]
Here the first map is multiplication on Milnor--Witt $\rmK$-theory.
\end{definition}


\begin{lemma}\label{prop:mod2MWHilb}
For any $v\in\Pl_F^\nc$, we have a commutative diagram
\[\begin{tikzcd}
\rmK_1^\MW(F)\times\rmK_1^\MW(F)\ar{r}{(-,-)_v^{\MW}}\ar{d}[swap]{p\times p} & B_v\ar{d}{q_v}\\
F^\times\times F^\times\ar{r}[swap]{(-,-)_v} & \mu(F_v).
\end{tikzcd}\]
\end{lemma}

\begin{proof}
The statement is clear for the finite places, so let $v$ be an infinite real place of $F$. Recall that $(a,b)_v\in\Z/2$ is defined as $0$ if $i_v(a)X^2+i_v(b)Y^2=1$ has a solution in $F_v\cong\R$, and $1$ otherwise \cite[p. 104]{Milnor}. On the other hand, the map \eqref{eq:signature-hom} carries $[a,b]$ to $0$ if any of $i_v(a)$ or $i_v(b)$ is positive in the ordering $v$, and to $1$ otherwise. So $([a],[b])_v^\MW\equiv(a,b)_v\pmod2$.
\end{proof}

\subsubsection{Wild kernels} Using the maps $h_v^\MW$ we can define wild kernels for Milnor--Witt $\rmK$-theory similarly as in the classical case:

\begin{definition}
Let $\WK^{\MW}_2(F)$ denote the kernel of the map $h^{\MW}=\bigoplus_{v\in\Pl_F^\nc}h_v^\MW$.
\end{definition}

\subsubsection{}We are now ready to formulate a Moore reciprocity sequence for Milnor--Witt $\rmK$-theory:

\begin{prop}\label{prop:MWmoore}
For any number field $F$, there is an isomorphism 
$\WK_2^\MW(F)\cong\WK_2(F).$ 
Moreover, there is an exact sequence
\[
0\to \WK_2(F)\to \rmK_2^{\MW}(F)\xrightarrow{h^{\MW}}\bigoplus_{v\in\Pl_F^\nc}B_v\to\mu(F)\to0,
\]
where $v$ runs over all noncomplex places of $F$, and where the $B_v$'s are the groups defined in \Cref{def:Bv}.
\end{prop}

\begin{proof}
Consider the following commutative diagram with exact rows:
\[
\begin{tikzcd}
& \ker p'\ar{d}\ar{r} & \rmI^3(F)\ar{d}\ar{r}{\iota}[swap]{\cong} & {\displaystyle\bigoplus_{v\in\Pl_F^\nc}}\rmI^3(F_v)\ar{r}\ar{d} & \ker q'\ar{d} & \\
0\ar{r} & \WK_2^{\MW}(F)\ar{r}\ar{d}{p'} & \rmK_2^{\MW}(F)\ar{r}{h^{\MW}}\ar{d}{p} & {\displaystyle\bigoplus_{v\in\Pl_F^\nc}}B_v\ar{r}\ar{d}{q} & \coker h^{\MW}\ar{d}{q'}\ar{r} & 0\\
0\ar{r} & \WK_2(F)\ar{r} & \rmK_2(F)\ar{r}{h} &{\displaystyle\bigoplus_{v\in\Pl_F^\nc}}\mu(F_v)\ar{r}{\pi} & \mu(F)\ar{r} & 0.
\end{tikzcd}
\]
Here $p'$ and $q'$ are induced from $p$ and $q$, respectively. According to \Cref{lemma:I3}, the map $\iota\colon\rmI^3(F)\to\bigoplus_{v\in\Pl_F^\nc}\rmI^3(F_v)$ is an isomorphism. It follows from this and a diagram chase that $p'$ and $q'$ are isomorphisms. 
%
%
\end{proof}

\begin{example}
We see from \Cref{prop:MW(Q)} that in the case $F=\Q$, the exact sequence of \Cref{prop:MWmoore} reads
\[
0\to\Z\oplus\bigoplus_{p\text{ prime}}\F_p^\times\to\Z\oplus\Z/2\oplus\bigoplus_{p\text{ prime}}\F_p^\times\to\Z/2\to0.
\]
Here we have used that $\WK_2(\Q)=0$.
\end{example}

\section{Regular kernels and Milnor--Witt \texorpdfstring{$\rmK$}{K}-theory of rings of integers}\label{section:K2}

\subsection{} We will now consider various kernels of the Hilbert symbols and the tame symbols. Recall that in classical $\rmK$-theory, there are three subgroups of $\rmK_2(F)$ of particular interest:
\[\begin{tikzcd}
\rmK_2(\calO_F)=\ker\p*{\partial\colon\rmK_2(F)\to\displaystyle{\bigoplus_{v\in\Pl_0}}k(v)^\times}\\
\rmK_2^{+}(\calO_F)\defeq\ker\p*{\displaystyle{\bigoplus_{v\in\Pl_\infty^\rmr}}h_v\big|_{\rmK_2(\calO_F)}\colon\rmK_2(\calO_F)\to\displaystyle{\bigoplus_{v\in\Pl_\infty^\rmr}}\mu(F_v)}\arrow[draw=none]{u}[sloped,auto=false]{\subseteq}\\
\WK_2(F)\defeq\ker\p*{h\colon\rmK_2(F)\to\displaystyle{\bigoplus_{v\in\Pl_F^\nc}}\mu(F_v)}\arrow[draw=none]{u}[sloped,auto=false]{\subseteq}.
\end{tikzcd}\]
We have already encountered the wild kernel $\WK_2(F)$, and it is a classical result of Quillen \cite[§5]{Quillen-higher-K} that $\rmK_2(\calO_F)$ is the kernel of the tame symbols $\partial$.
The group $\rmK_2^{+}(\calO_F)$ was introduced by Gras in \cite{Gras-reg} and is referred to as the \emph{regular kernel}, or the \emph{narrow $\rmK_2$-group}. It is a modification of $\rmK_2(\calO_F)$ that takes into account also the real places of $F$.

\begin{definition}\label{def:reg-kernels}
For any $n\ge1$,
let $\rmK_n^\MW(\calO_F)$ denote the kernel of 
\[
\partial\defeq\bigoplus_{v\in\Pl_0}\partial_v^{\pi_v}\colon\rmK_n^\MW(F)\to\bigoplus_{v\in\Pl_0}\rmK_{n-1}^\MW(k(v)).
\]
Here $\pi_v$ is any choice of uniformizer for the discrete valuation $v$.
Moreover, we let
\[
\rmK_n^{\MW,+}(\calO_F)\defeq\ker\p*{\rmK_n^\MW(\calO_F)\to\Z^{r_1}};\quad
\rmK_n^{\MW,+}(F)\defeq\ker\p*{\rmK_n^\MW(F)\to\Z^{r_1}},
\]
where the maps are given by the signatures with respect to the orderings on $F$, as defined in \eqref{eq:signature-hom}.
\end{definition}

\begin{remark}
By \cite[Proposition 3.17 (3)]{Morel-over-a-field}, the kernel of $\partial_v^{\pi_v}$ is independent of the choice of uniformizer $\pi_v$.
\end{remark}

\begin{example}
By \Cref{prop:MW(Q)} we have $\rmK_n^\MW(\Z)\cong\Z$ for all $n\ge1$.
\end{example}

\begin{remark}
If $S$ is a set of places of $F$ containing the infinite places, note that we can also define Milnor--Witt $\rmK$-theory of the ring of $S$-integers $\calO_{F,S}$ in $F$ as
\[
\rmK_n^\MW(\calO_{F,S})\defeq\ker\p*{\bigoplus_{v\not\in S}\partial_v^{\pi_v}\colon\rmK_n^\MW(F)\to\bigoplus_{v\not\in S}\rmK_{n-1}^\MW(k(v))}.
\]
Here $\pi_v$ is a uniformizer for the place $v$. For $n=2$, the groups $\rmK_2^\MW(\calO_{F,S})$ were also considered in \cite{Hutchinson-SL2}. More precisely, Hutchinson defines a subgroup $\wt\rmK_2(2,\calO_{F,S})$ of the second unstable $\rmK$-group $\rmK_2(2,F)$ by
\[
\wt\rmK_2(2,\calO_{F,S})\defeq\ker\p*{\rmK_2(2,F)\to\bigoplus_{v\not\in S}k(v)^\times}.
\]
By \cite[Proposition 3.12]{Hutchinson-SL2} there is a natural isomorphism $\rmK_2(2,F)\xrightarrow{\cong}\rmK_2^\MW(F)$, hence $\wt\rmK_2(2,\calO_{F,S})\cong \rmK_2^{\MW}(\calO_{F,S})$. The main theorem of \cite{Hutchinson-SL2} states that if $S$ is a set of places of $\Q$ containing $2$ and $3$, then $\rmK_2^\MW(\Z_S)\cong\rmH_2(\SL_2(\Z_S),\Z)$ (where $\Z_S$ denotes the localization of $\Z$ at the primes of $S$). In particular, $\rmK_2^\MW(\Z[1/m])\cong\rmH_2(\SL_2(\Z[1/m]),\Z)$ provided $6\mid m$. It is a conjecture that this isomorphism holds for any even $m$.
\end{remark}

\begin{prop}\label{prop:KMWOF}
\begin{enumerate}
\item[(i)] We have a short exact sequence
\[
0\to\Z^{r_1}\to\rmK_2^{\MW}(\calO_F)\xrightarrow{q}\rmK_2(\calO_F)\to0,
\]
where the homomorphism $q$ is induced by the forgetful map from Milnor--Witt $\rmK$-theory to Milnor $\rmK$-theory.
\item[(ii)] We have $\rmK_2^{\MW,+}(\calO_F)\cong\rmK_2^{+}(\calO_F)$ and $\rmK_2^{\MW,+}(F)\cong\rmK_2^{+}(F)$.\label{eq:refone}
\item[(iii)] The groups $\rmK_2^{\MW}(\calO_F)$ and $\rmK_2^{\MW}(F)$ decompose as\label{eq:reftwo}
\[
\rmK_2^{\MW}(\calO_F)\cong\rmK_2^{+}(\calO_F)\oplus\Z^{r_1};\quad\rmK_2^{\MW}(F)\cong\rmK_2^{+}(F)\oplus\Z^{r_1}.
\]
\end{enumerate}
In particular, by Garland's theorem on the finiteness of $\rmK_2(\calO_F)$ \cite{Garland}, the group $\rmK_2^\MW(\calO_F)$ is a finitely generated abelian group of rank $r_1$.
\end{prop}

\begin{proof}
For the first point, consider the diagram
\[\begin{tikzcd}
0\ar{r} & \rmK_2^\MW(\calO_F)\ar{r}\ar{d}{q} & \rmK_2^\MW(F)\ar{d}{p}\ar{r}{\partial} & \displaystyle{\bigoplus_{v\in\Pl_0}}\rmK_1^\MW(k(v))\ar{d}{\cong}\ar{r} & 0\\
0\ar{r} & \rmK_2(\calO_F)\ar{r} & \rmK_2(F)\ar{r}{\partial} & \displaystyle{\bigoplus_{v\in\Pl_0}}k(v)^\times\ar{r} & 0.
\end{tikzcd}\]
By \Cref{prop:K2MW(Fv)} (i), the right hand vertical map is an isomorphism. Using \Cref{lemma:I3}, the claim follows.

For the second point, first note that we have a commutative diagram with exact rows
\begin{equation}
\begin{tikzcd}\label{eq:diagr-K2MWplus}
0\ar{r} & \rmK_2^{\MW,+}(\calO_F)\ar{r}\ar{d} & \rmK_2^\MW(\calO_F)\ar{r}\ar{d} & \Z^{r_1}\ar{r}\ar{d}& 0\\
0\ar{r} & \rmK_2^{+}(\calO_F)\ar{r}      & \rmK_2(\calO_F)\ar{r}     & (\Z/2)^{r_1}\ar{r} &  0.
\end{tikzcd}
\end{equation}
Here the upper right hand-side map is given by the signature homomorphisms \eqref{eq:signature-hom}; commutativity of the right hand square follows similarly as the proof of \Cref{prop:mod2MWHilb}.
The first claim of (ii) then follows from (i) along with the snake lemma applied to the diagram \eqref{eq:diagr-K2MWplus}; the second claim of (ii) follows similarly.

For (iii), we use (ii) along with the observation that since the right hand-side in the upper short exact sequence of \eqref{eq:diagr-K2MWplus}
is free, the sequence splits. A similar argument works for $\rmK_2^{\MW}(F)$.
\end{proof}

\begin{remark}
In contrast to point (iii) of \Cref{prop:KMWOF}, the corresponding short exact sequence for $\rmK_2(\calO_F)$,
\[
0\to\rmK_2^{+}(\calO_F)\to\rmK_2(\calO_F)\to(\Z/2)^{r_1}\to0,
\]
need not split. For example, \cite[Example 3.10]{Keune-K2} shows that if $F=\Q(\sqrt{14})$, then $\rk_2(\rmK_2^{+}(\calO_F))=1$ while $\rk_2(\rmK_2(\calO_F))=2$. On the other hand, the sequence
\[
0\to\rmK_2^{+}(F)\to\rmK_2(F)\to(\Z/2)^{r_1}\to0
\]
is always split (see \cite[§2.1]{Keune-K2}).
\end{remark}

\subsection{Sample computations of \texorpdfstring{$\rmK_2^\MW(\calO_F)$}{K2MW(OF)}}
Using \Cref{prop:KMWOF} along with similar results for $\rmK_2(\calO_F)$ we deduce some computations of $\rmK_2^\MW(\calO_F)$ for various fields $F$. Below we make use of the calculations of \cite{Belabas} to determine $\rmK_2^\MW(\calO_F)$. 

In the following table, we consider the number fields $F=\Q[x]/(P)$ defined by the polynomial $P$. We let $\Delta_F$ denote the discriminant of $F$, and $(r_1,r_2)$ the signature. 
\begin{center}
\begin{tabular}{l|rlll}
\toprule
 $P$ & $\Delta_F$  & $(r_1,r_2)$  & $\rmK_2(\calO_F)$  & $\rmK_2^\MW(\calO_F)$  \\
\midrule
$x^3+x^2-2x-1$   & $49$ & $(3,0)$  & $(\Z/2)^3$  & $\Z^3$ \\
$x^3+x^2+2x+1$   & $-23$ & $(1,1)$  & $\phantom{(}\Z/2$  & $\Z$ \\
$x^3+x^2+3$      & $-255$ & $(1,1)$ & $\phantom{(}\Z/6$ & $\Z/3\oplus\Z$\\
$x^4-x-1$        & $-283$ & $(2,1)$ & $(\Z/2)^2$ & $\Z^2$ \\
$x^5-x^3-x^2+x+1$ & $1609$ & $(1,2)$ & $\phantom{(}\Z/2$ & $\Z$ \\
\bottomrule
\end{tabular}
\end{center}

\section{Hasse's norm theorem for \texorpdfstring{$\rmK_2^\MW$}{K2MW}}\label{section:Hasse}

Recall that the classical norm theorem of Hasse states that if $L$ is a cyclic extension of the number field $F$, then a nonzero element of $F$ is a local norm at every place if and only if it is a global norm. We can think of this result as a norm theorem for $\rmK_1$.

In \cite{Bak-Rehmann}, Bak and Rehmann extended Hasse's norm theorem to $\rmK_2$. More precisely, they showed that if $L/F$ is any finite extension of number fields, then an element of $\rmK_2(F)$ lies in the image of the transfer map $\tau_{L/F}\colon\rmK_2(L)\to\rmK_2(F)$ if and only if its image in each $\rmK_2(F_v)$ lies in the image of the map 
\[
\bigoplus_{w\mid v}\tau_{L_w/F_v}\colon\bigoplus_{w\mid v}\rmK_2(L_w)\to\rmK_2(F_v);
\]
see \cite[Theorem 1]{Bak-Rehmann}. By \cite[p. 4]{Bak-Rehmann}, this result can be reformulated as the exactness of the sequence
\begin{align}
\rmK_2(L)\xrightarrow{\tau_{L/F}}\rmK_2(F)\xrightarrow{\bigoplus_{v\in\Sigma_{L/F}}h_v}\bigoplus_{v\in\Sigma_{L/F}}\mu(F_v)\to0.\label{exseq:norms}
\end{align}
Here $\Sigma_{L/F}$ denotes the set of infinite real places of $F$ that are complexified in the extension $L/F$, and $h_v$ denotes the classical local Hilbert symbol at $v$.

\subsubsection{}The aim of this section is to show that a similar result as \eqref{exseq:norms} holds also for $\rmK_2^\MW$:

\begin{prop}\label{prop:norm1}
Let $L/F$ be an extension of number fields. Then an element of $\rmK_2^\MW(F)$ lies in the image of the transfer map $\tau_{L/F}\colon\rmK_2^\MW(L)\to\rmK_2^\MW(F)$ if and only if its image in $\rmK_2^\MW(F_v)$ lies in the image of $\bigoplus_{w\mid v}\tau_{L_w/F_v}$ for all $v\in\Pl_F$.
\end{prop}

\subsubsection{}We note that \Cref{prop:norm1} is equivalent to the following assertion: 

\begin{prop}\label{prop:norm}
Let $L/F$ be an extension of number fields. Denote by $\Sigma_{L/F}$ the set of infinite real places of $F$ that are complexified in the extension $L/F$. Then there is an exact sequence
\[
\rmK_2^\MW(L)\xrightarrow{\tau_{L/F}}\rmK_2^\MW(F)\xrightarrow{\bigoplus_{v\in\Sigma_{L/F}}h_v^\MW}\bigoplus_{v\in\Sigma_{L/F}}\Z\to0,
\]
where $\tau_{L/F}$ denotes the transfer map on Milnor--Witt $\rmK$-theory as defined in \eqref{transfer-map}, and the right hand-side map is given by the local Milnor--Witt Hilbert symbols.
\end{prop}

To show the equivalence of Propositions \ref{prop:norm1} and \ref{prop:norm}, notice first that \Cref{prop:norm1} is equivalent to the assertion that the map $\coker(\tau_{L/F})\to\prod_{v\in\Pl_F}\coker\p*{\bigoplus_{w\mid v}\tau_{L_w/F_v}}$ is injective. It is therefore enough to show that $\coker(\tau_{L_w/F_v})$ is trivial except for $v\in\Sigma_{L/F}$, in which case it is $\Z$. But this follows from \Cref{prop:K2MW(Fv)}  along with the corresponding statement for $\rmK_2$ proved in \cite[p. 4]{Bak-Rehmann}.

\subsubsection{}
Let us proceed with the proof of \Cref{prop:norm}. We follow closely the strategy of Bak and Rehmann \cite{Bak-Rehmann}. Recall from \Cref{def:Bv} the definition of the groups $B_v$ for $v\in\Pl_F$.

\begin{definition}
If $v$ is any place of $F$ and $w\mid v$ is a place of $L$ above $v$, we define a homomorphism
\[
b_{w\mid v}\colon B_w\to B_v
\]
by
\[
b_{w\mid v}\defeq\begin{cases}
0, & v\in\Pl_{F,\infty}^\rmr\text{ and }w\in\Pl_{L,\infty}^\rmc\\
\id, & v\in\Pl_{F,\infty}^\rmr\text{ and }w\in\Pl_{L,\infty}^\rmr\\
n_w/m_v, & v\in\Pl_{F,0}.
\end{cases}
\]
Here $n_w\defeq\#\mu(L_w)$ and $m_v\defeq\#\mu(F_v)$.
\end{definition}

\begin{lemma}
For any place $v$ of $F$, the diagram
\[\begin{tikzcd}
\rmK_2^\MW(L)\ar{d}[swap]{\tau_{L/F}}\ar{r}{\bigoplus_{w\mid v}h^\MW_w} & \displaystyle{\bigoplus_{w\mid v}B_w}\ar{d}{\bigoplus_{w\mid v}b_{w\mid v}}\\
\rmK_2^\MW(F)\ar{r}{h^\MW_v} & B_v
\end{tikzcd}\]
is commutative.
\end{lemma}

\begin{proof}
Since $\rmK_2^\MW(F_v)\cong\rmK_2(F_v)$ for each $v\not\in\Pl_\infty^\rmr$ it suffices by \cite[Proof of Proposition 2]{Bak-Rehmann} to note that 
$
\tau_{L_w/F_v}=\id\colon\rmK_2^{\MW}(\R)\to\rmK_2^\MW(\R)
$ 
whenever $v$ and $w$ both are infinite real places. Indeed, this follows since $L_w/F_v$ is then the trivial extension.
\end{proof}

\begin{lemma}
Let $x\in\WK^\MW_2(F)$, and let $p$ be a rational prime. Then there is an element $z\in\rmK_2^\MW(F)$ such that $pz=x$.
\end{lemma}

\begin{proof}
By \Cref{prop:MWmoore} we have $\WK_2^\MW(F)\cong\WK_2(F)$. Moreover, $\rmK_2^\MW(F)\cong\rmK_2^{+}(F)\oplus\Z^{r_1}$ by \Cref{prop:KMWOF}. Then, since $\WK_2^\MW(F)\subseteq\rmK_2^{+}(F)$, the statement follows from the fact that any element in the classical wild kernel has a $p$-th root in $\rmK_2(F)$ \cite{Tate-symbols}. 
\end{proof}

\begin{lemma}\label{lemma:roots}
Suppose that $x\in\ker\p*{\bigoplus_{v\in\Sigma_{L/F}}h_v^\MW}$.
If $nz=x$ for some $z\in\rmK_2^\MW(F)$ and some integer $n\ge1$, then there is an element $y\in\rmK_2^\MW(L)$ such that $x+\tau_{L/F}(y)\in n\WK_2^\MW(F)$.
\end{lemma}

\begin{proof}
Using the Moore reciprocity sequence \ref{prop:MWmoore} along with the fact that the homomorphism $\rmK_2^\MW(F)\to\bigoplus_{v\in\Sigma_{L/F}}\Z$
is split,
 the same proof as that of \cite[Lemma 1]{Bak-Rehmann} goes through.
\end{proof}

\begin{proof}[Proof of \Cref{prop:norm}]
The desired result now follows from \Cref{lemma:roots} in an identical manner as the proof of \cite[Theorem 2]{Bak-Rehmann}.
\end{proof}

\printbibliography
\end{document}